\documentclass{amsproc}

\DeclareMathSymbol{\N}{\mathbin}{AMSb}{"4E}
\DeclareMathSymbol{\Z}{\mathbin}{AMSb}{"5A}

\newtheorem{theorem}{Theorem}[section]
\newtheorem{lemma}[theorem]{Lemma}
\newtheorem{proposition}[theorem]{Proposition}

\theoremstyle{definition}

\theoremstyle{remark}

\numberwithin{equation}{section}

\begin{document}


\title{The Inverse Problem for Canonically Bounded Rank-one Transformations}

\author{Aaron Hill}
\address{Department of Mathematics\\ University of Louisville\\  Louisville, KY 40292}
\email{aaron.hill@louisville.edu}
\thanks{The author acknowledges the US NSF grant DMS-0943870 for the support of his research.}

\date{July 3, 2015 and, in revised form, December 11, 2015.}


\begin{abstract}
Given the cutting and spacer parameters for a rank-1 transformation, there is a simple condition which is easily seen to be sufficient to guarantee that the transformation under consideration is isomorphic to its inverse.  Here we show that if the cutting and spacer parameters are canonically bounded, that condition is also necessary, thus giving a simple characterization of the canonically bounded rank-1 transformations that are isomorphic to their inverse.  
\end{abstract}

\maketitle

\section{Introduction}

\subsection{Background}
A measure-preserving transformation is an automorphism of a standard Lebesgue space.  Formally, it is a quadruple $(X, \mathcal{B}, \mu, T)$, where 
\begin{enumerate}
\item  $(X, \mathcal{B}, \mu)$ is a measure space isomorphic to the unit interval with the Lebesgue measure on all Borel sets,
\item $T$ is a bijection from $X$ to $X$ such that $T$ and $T^{-1}$ are both $\mu$-measurable and preserve the measure $\mu$.
\end{enumerate}
When the algebra of measurable sets is clear, we will refer to the transformation $(X, \mathcal{B}, \mu, T)$ by $(X, \mu, T)$.  If $(X, \mathcal{B}, \mu, T)$ is a measure-preserving transformation, then so is its inverse, $(X, \mathcal{B}, \mu, T^{-1})$.

Two measure-preserving transformations $(X, \mathcal{B}, \mu, T)$ and $(X^\prime, \mathcal{B}^\prime, \mu^\prime, T^\prime)$ are isomorphic if there is a measure isomorphism $\phi$ from  $(X, \mathcal{B}, \mu)$ to $(X^\prime, \mathcal{B}^\prime, \mu^\prime)$ such that $\mu$ almost everywhere, $\phi \circ T = T^\prime \circ \phi $.    

One of the central problems of ergodic theory, originally posed by von Neumann, is the isomorphism problem:  How can one determine whether two measure-preserving transformations are isomorphic?  The inverse problem is one of its natural restrictions: How can one determine whether a measure-preserving transformation is isomorphic to its inverse?  

In the early 1940s, Halmos and von Neumann \cite{HalmosvonNeumann} showed that ergodic measure-preserving transformations with pure point spectrum are isomorphic iff they have the same spectrum.  It immediately follows from this that every ergodic measure-preserving transformation with pure point spectrum is isomorphic to its inverse.  
About a decade later, Anzai \cite{Anzai} gave the first example of a measure-preserving transformation not isomorphic to its inverse.  Later, Fieldsteel \cite{Fieldsteel} and del Junco, Rahe, and Swanson \cite{delJuncoRaheSwanson} independently showed that Chacon's transformation--one of the earliest examples of what we now call rank-1 transformations--is not isomorphic to its inverse.  In the late 1980s, Ageev \cite{Ageev3} showed that a generic measure-preserving transformation is not isomorphic to its inverse.  

In 2011, Foreman, Rudolph, and Weiss \cite{ForemanRudolphWeiss} showed that the set of ergodic measure-preserving transformations of a fixed standard Lebesgue space that are isomorphic to their inverse is a complete analytic subset of all measure-preserving transformations on that space.  In essence, this result shows that there is no simple (i.e., Borel) condition which is satisfied if and only if an ergodic measure-preserving transformation is isomorphic to its inverse.  However, in the same paper they show that the isomorphism relation becomes much simpler when restricted to the generic class of rank-1 transformations.  It follows from their work that there exists a simple (i.e., Borel) condition which is satisfied if and only if a rank-1 measure-preserving transformation is isomorphic to its inverse.  Currently, however, no such condition is known.  In this paper we give a simple condition that is sufficient for a rank-1 transformation to be isomorphic to its inverse and show that for canonically bounded rank-1 transformations, the condition is also necessary.

\subsection{Rank-1 transformations}
\label{comments}


In this subsection we state the definitions and basic facts pertaining to rank-1 transformations that will be used in our main arguments.  
We mostly follow the symbolic presentation in \cite{GaoHill1} and \cite{GaoHill2}, but also provide comments that hopefully will be helpful to those more familiar with a different approach to rank-1 transformations.  Additional information about the connections between different approaches to rank-1 transformations can be found in the survey article \cite{Ferenczi}.  

We first remark that by $\N$ we mean the set of all finite ordinals, including zero:  $\{0, 1, 2, \ldots \}$.

Our main objects of study are symbolic rank-1 measure-preserving transformations.  Each such transformation is a measure-preserving transformation $(X, \mathcal{B}, \mu, \sigma)$, where $X$ is a closed, shift-invariant subset of $\{0,1\}^\Z$, $\mathcal{B}$ is the collection of Borel sets that $X$ inherits from the product topology on $\{0,1\}^\Z$, $\mu$ is an atomless, shift-invariant (Borel) probability measure on $X$, and $\sigma$ is the shift.  To be precise, the shift $\sigma$ is the bijection from $\{0,1\}^\Z$ to $\{0,1\}^\Z$, where $\sigma (x) (i) = x (i+1)$.  Since the measure algebra of a symbolic measure-preserving transformation comes from the topology on $\{0,1\}^\Z$, we will omit the reference to that measure algebra and simply refer to a symbolic measure-preserving transformation as $(X, \mu, \sigma)$.

Symbolic rank-1 measure-preserving transformations are usually described by {\em cutting and spacer parameters}.  The cutting parameter is a sequence $(r_n : n \in \N)$ of integers greater than 1.  The spacer parameter is a sequence of tuples $(s_n : n \in \N)$, where formally $s_n$ is a function from $\{1, 2, \ldots, r_n -1\}$ to $\N$ (note that $s_n$ is allowed to take the value zero).  Given 
such cutting and spacer parameters, one defines the symbolic rank-1 system $(X, \sigma)$ as follows.  First define a sequence of finite words $(v_n : n \in \N)$ by $v_0 =0$ and $$v_{n+1} = v_n 1^{s_n(1)} v_n 1^{s_n(2)}v_n \ldots  v_n 1^{s_n(r_n-1)} v_n.$$  The sequence $(v_n: n \in \N)$ is called a {\em generating sequence}.  Then let $$X = \{x \in \{0,1\}^\Z: \text{ every finite subword of $x$ is a subword of some $v_n$}\}.$$  It is straightforward to check that $X$ is a closed, shift-invariant subset of $\{0,1\}^\Z$.  These symbolic rank-1 systems are treated extensively--as topological dynamical systems--in \cite{GaoHill1}.  In order to introduce a nice measure $\mu$ and thus obtain a measurable dynamical system, we make two additional assumptions on the cutting and spacer parameters.
\begin{enumerate}
\item  For every $N \in \N$ there exist $n, n^\prime \geq N$ and $0 < i < r_n$ and $0 < i^\prime < r_{n^\prime}$ such that $$s_n(i) \neq s_{n^\prime} (i^\prime).$$
\item   $\displaystyle \sup_{n \in \N}  \frac{ \text{\# of 1s in $v_n$}}{|v_n|} < 1  $
\end{enumerate}
It is straightforward to show that there is a unique shift-invariant measure on $X$ which assigns measure 1 to the set $\{x \in X: x (0) = 0\}$.  As long as the first condition above is satisfied, that measure is atomless.  As long as the second condition above is satisfied, that measure is finite.  Assuming both conditions are satisfied, the normalization of that measure is called $\mu$ and then $(X, \mu, \sigma)$ is a measure-preserving transformation.  We call such an $(X, \mu, \sigma)$ a symbolic rank-1 measure-preserving transformation.

Below are several important remarks about symbolic rank-1 measure-preserving transformations that will be helpful in understanding the arguments in Section 2.
\begin{itemize}

\item  {\em Bounded rank-1 transformations:}  Suppose $(r_n: n \in \N)$ and $(s_n: n \in \N)$ are cutting and spacer parameters for $(X, \mu, \sigma)$.  We say the cutting parameter is bounded if there is some $R \in \N$ such that for all $n \in \N$, $r_n \leq R$.  We say that the spacer parameter is bounded if there is some $S \in \N$ such that for all $n \in \N$ and all $0<i < r_n$, $s_n(i) \leq S$.

Let $(X, \mu, \sigma)$ be a symbolic rank-1 measure-preserving transformation.  We say that $(X, \mu, \sigma)$ is bounded if there are cutting and spacer parameters $(r_n: n \in \N)$ and $(s_n: n \in \N)$ that give rise to $(X, \mu, \sigma)$ that are both bounded. 

\item {\em Canonical cutting and spacer parameters:} There is an obvious bijective correspondence between cutting and spacer parameters and generating sequences, but there are many different generating sequences that give rise to the same symbolic rank-1 system.  For example, any proper subsequence of a generating sequence will be a different generating sequence that gives rise to the same symbolic rank-1 system.  There is a way, however, described in \cite{GaoHill1}, to associate to each symbolic rank-1 system a unique canonical generating sequence, which in turn gives rise to the canonical cutting and spacer parameters of that symbolic system.  The canonical generating sequence was used in \cite{GaoHill1} to fully understand topological isomorphisms between symbolic rank-1 systems; it was also used in \cite{GaoHill2} to explicitly describe when a bounded rank-1 measure-preserving transformation has trivial centralizer.

There is only one fact that about canonical generating sequences that is used in our argument.  It is this: If $(r_n: n \in \N)$ and $(s_n: n \in \N)$ are the canonical cutting and spacer parameters for $(X, \mu, \sigma)$, then for all $n \in \N$, there is $0<i<r_n$ and $0<j<r_{n+1}$ such that $s_n(i) \neq s_n(j)$.  (See the definition of canonical generating sequence in section 2.3.2 and 2.3.3 in \cite{GaoHill1})

\item {\em Expected occurrences:} Let $(v_n: n \in \N)$ be a generating sequence giving rise to the symbolic system $(X, \sigma)$.   Then for each $n \in \N$, there is a unique way to view each $x \in X$ as a disjoint collection of occurrences of $v_n$ separated only by 1s.  Such occurrences of $v_n$ in $x$ are called {\em expected} and the following all hold.
\begin{enumerate}
\item  For all $x \in X$ and $n \in \N$, every occurrence of $0$ in $x$ is contained in a unique expected occurrence of $v_n$. 
\item  For all $x \in X$ and $n \in \N$, $x$ has an expected occurrence of $v_n$ beginning at position $i$ iff $\sigma (x)$ has an expected occurrence of $v_n$ beginning at position $(i-1)$.
\item  If $x \in X$ has an expected occurrence of $v_n$ beginning at position $i$,  and $n^\prime > n$, then the unique expected occurrence of $v_{n^\prime}$ that contains the 0 at position $i$ completely contains the expected occurrence of $v_n$ that begins at $i$.
\item  If $x \in X$ has expected occurrences of $v_n$ beginning at positions $i$ and $j$, with $|i - j| < |v_n|$, then $i=j$. In other words, distinct expected occurrences of $v_n$ cannot overlap.
\item  If $n>m$ and $x\in X$ has as expected occurrence of $v_n$ beginning at $i$ which completely contains an expected occurrence of $v_m$ beginning at $i + l$, then whenever $j$ is such that $x$ has an expected occurrence of $v_n$ beginning at $j$, that occurrence completely contains an expected occurrence of $v_m$ beginning at $j + l$.
\end{enumerate} 
For $n \in \N$ and $i \in \Z$ we define $E_{v_n,i}$ to be the set of all $x \in X$ that have an expected occurrence of $v_n$ beginning at position $i$.

\item  {\em Relation to cutting and stacking constructions:}  Let $(v_n: n \in \N)$ be a generating sequence giving rise to the symbolic rank-1 measure-preserving system $(X, \mu, \sigma)$.  One can take the cutting and spacer parameters associated to $(v_n: n \in \N)$ and build, using a cutting and stacking construction, a rank-1 measure-preserving transformation. This construction involves a sequence of Rokhlin towers.  There is a direct correspondence between the base of the $n$th tower in the cutting and stacking construction and the set $E_{v_n, 0}$ in the symbolic system.  The height of the $n$th tower in the cutting and stacking construction then corresponds to (i.e., is equal to) the length of the word $v_n$.  If the reader is more familiar with rank-1 transformations as cutting and stacking constructions, one can use this correspondence to translate the arguments in Section 2 to that setting.

\item  {\em Expectedness and the measure algebra:}  Let $(v_n: n \in \N)$ be a generating sequence giving rise to the symbolic rank-1 measure-preserving system $(X, \mu, \sigma)$.  If $\mathbb{M}$ is any infinite subset of $\N$, then the collection of sets $\{E_{v_n, i}: n \in \mathbb{M}, i \in \Z\}$ is dense in the measure algebra of $(X, \mu)$.  Thus if $A$ is any positive measure set and $\epsilon >0$, there is some $n \in \mathbb{M}$ and $i \in \Z$ such that $$\frac{\mu (E_{v_n, i} \cap A)}{\mu(E_{v_n, i}) } > 1 - \epsilon$$

\item {\em Rank-1 Inverses:} Let $(r_n: n \in \N)$ and $(s_n: n \in \N)$ be cutting and spacer parameters for the symbolic rank-1 measure-preserving transformation $(X, \mu, \sigma)$.  It is straightforward to check that a simple modification of the parameters results in a symbolic rank-1 measure-preserving transformation that is isomorphic to $(X, \mu, \sigma^{-1})$.  For each tuple $s_n$ in the spacer parameter, let $\overline{s_n}$ be the reverse tuple, i.e., for $0 < i < r_n$, $\overline{s_n } (i) = s_n (r_n -i)$. It is easy to check that the cutting and spacer parameters $(r_n: n \in \N)$ and $(\overline{s_n}: n \in \N)$ satisfy the two measure conditions necessary to produce a symbolic rank-1 measure-preserving transformation.  If one denotes that transformation by  $(\overline{X}, \overline{\mu}, \sigma)$  and defines $\phi : X \rightarrow \overline{X}$ by $\psi (x) (i) = x(-i)$, then it is straightforward to check that $\psi$ is an isomorphism between  $(X, \mu, \sigma^{-1})$ and $(\overline{X}, \overline{\mu}, \sigma)$.  Thus to check whether a given symbolic rank-1 measure-preserving transformation $(X, \mu, \sigma)$ is isomorphic to its inverse, one need only check whether it is isomorphic to the symbolic rank-1 measure-preserving transformation $(\overline{X}, \overline{\mu}, \sigma)$.
  
 \end{itemize}
 
 \subsection{The condition for isomorphism and the statement of the theorem}

Let $(r_n: n \in \N)$ and $(s_n: n \in \N)$ be cutting and spacer parameters for the symbolic rank-1 measure-preserving transformation $(X, \mu, \sigma)$.  Suppose that there is an $N \in \N$ such that for all $n \geq N$, $s_n = \overline{s_n}$.  Let $\phi : X \rightarrow \overline{X}$ be defined so that $\phi (x)$ is obtained from $x$ by replacing every expected occurrence of $v_N$ by $\overline{v_N}$ (the reverse of $v_N$).  It is straightforward to check that $\phi$ is an isomorphism between $(X, \mu, \sigma)$ and $(\overline{X}, \overline{\mu}, \sigma)$, thus showing that $(X, \mu, \sigma)$ is isomorphic to its inverse $(X, \mu, \sigma^{-1})$.  

As an example, Chacon2 is the rank-one transformation that can be defined by $v_{n+1} = v_n 1^n v_n$.   (In the cutting and stacking setting, Chacon2 is usually described by $B_{n+1} = B_n B_n 1$, but that is easily seen to be equivalent to $B_{n+1} = B_n 1^n B_n$.)  In this case $r_n = 2$ and $s_n(1)=n$, for all $n$.  Since $s_n = \overline{s_n}$ for all $n$, Chacon2 is isomorphic to its inverse.

\begin{theorem}
\label{theorem}
Let $(r_n: n \in \N)$ and $(s_n: n \in \N)$ be the canonical cutting and spacer parameters for the symbolic rank-1 measure-preserving transformation $(X, \mu, \sigma)$.  If those parameters are bounded, then $(X, \mu, \sigma)$ is isomorphic to $(X, \mu, \sigma^{-1})$ if and only if there is an $N \in \N$ such that for all $n \geq N$, $s_n = \overline{s_n}$.
\end{theorem}

We remark that in \cite{GaoHill1}, the author and Su Gao have completely characterized when two symbolic rank-1 system are {\em topologically} isomorphic, and as a corollary have a complete characterization of when a symbolic rank-1 system is {\em topologically} isomorphic to its inverse.   A topological isomorphism between symbolic rank-1 systems is a homeomorphism between the underlying spaces that commutes with the shift.  Since each the underlying space of a symbolic rank-1 system admits at most one atomless, shift-invariant probability measure, every topological isomorphism between symbolic rank-1 systems is also a measure-theoretic isomorphism.  On the other hand, there are symbolic rank-1 systems that are measure-theoretically isomorphic, but not topologically isomorphic.



We note here the main difference between these two settings.  Suppose $\phi$ is an isomorphism--either a measure-theoretic isomorphism or a topological isomorphism--between two symbolic rank-1 systems $(X, \mu, \sigma)$ and $(Y, \nu, \sigma)$.  
Let $(v_n : n \in \N)$ and $(w_n : n \in \N)$ be generating sequences that gives rise to $(X, \mu, \sigma)$ and $(Y, \nu, \sigma)$, respectively.
One can consider a set $E_{w_m, 0} \subseteq Y$ and its pre-image, call it $A$, under $\phi$.  If $\phi$ is a measure-theoretic isomorphism then one can find some $E_{v_n, i}$ so that $$\frac{\mu (E_{v_n, i} \cap A)}{\mu(E_{v_n, i}) } > 1 - \epsilon.$$
However, if $\phi$ is in fact a topological isomorphism, then one can find some $E_{v_n, i}$ so that $$E_{v_n, i} \subseteq A.$$
The stronger condition in the case of a topological isomorphism is what makes possible the analysis done by the author and Gao in \cite{GaoHill1}.  In this paper, we are able to use the weaker condition, together with certain ``bounded" conditions on the generating sequences $(v_n : n \in \N)$ and $(w_n: n \in \N)$ to achieve our results.

 


\section{Arguments}

We begin with a short subsection introducing two new pieces of notation.  Then we prove a general proposition that can be used to show that certain symbolic rank-1 measure-preserving transformations are not isomorphic.  Finally, we show how to use the general proposition to prove the non-trivial direction of Theorem \ref{theorem}.

\subsection{New notation}

The first new piece of notation is $*$, a binary operation on all finite sequences of natural numbers.  The second is $\perp$, a relation (signifying incompatibility) between finite sequences of natural numbers that have the same length. 

{\bf The notation $*$:}  We will first describe the reason for introducing this new notation.  We will then then give the formal definition of $*$ and then illustrate that definition with an example.  Suppose $(r_n :n \in \N)$ and $(s_n: n \in \N)$ are cutting and spacer parameters for the symbolic system $(X, \mu, \sigma)$ and that $(v_n: n \in \N)$ is the generating sequence corresponding to those parameters.  Fix $n_0 > 0$ and consider the generating sequence $(w_n : n \in \N)$, defined as follows.
$$w_n = 
\begin{cases}
v_n, \quad &\text{ if } n < n_0 \\
v_{n+1}, \quad &\text{ if }  n \geq n_0\\
\end{cases}
$$
It is clear that $(w_n: n \in \N)$ is a subsequence of $(v_n: n \in \N)$, missing only the element $v_{n_0}$; thus, $(w_n: n \in \N)$ gives rise to the same symbolic system $(X, \mu, \sigma)$.  We would like to be able to easily describe the cutting and spacer parameters that correspond to the generating sequence $(w_n:n \in \N)$.  Let $(r_n^\prime :n \in \N)$ and $(s_n^\prime : n \in \N)$ be those cutting and spacer parameters.  It is clear that for $n< n_0$ we have $r_n^\prime  = r_n$ and $s_n^\prime  = s_n$.  It is also clear that for $n>n_0$ we have $r_n^\prime  = r_{n+1}$ and $s_n^\prime  = s_{n+1}$.  It is straightforward to check that $r_{n_0}^\prime = r_{n_0+1} \cdot r_{n_0}$.  The definition below for $*$ is precisely what is needed so that $s_{n_0}^\prime = s_{n_0+1} * s_{n_0}$.

Here is the definition. Let $s_1$ be any function from $\{1, 2, \ldots, r_1 -1\}$ to $\N$ and let $s_2$ be any function from $\{1, 2, \ldots, r_2 -1\}$ to $\N$.  We define $s_2 * s_1$, a function from $\{1, 2, \ldots,  r_2 \cdot r_1 -1\}$ to $\N$, as follows.
$$(s_2 *s_1 )(i) = 
\begin{cases}
s_1(k), \quad &\text{ if } 0 < k < r_1  \text{ and } i \equiv k \mod r_1 \\
s_2(i/r_1), \quad &\text{ if } i \equiv 0 \mod r_1 \\
\end{cases}
$$
It is important to note, and straightforward to check, that the operation $*$ is associative. 

To illustrate, suppose that $s_1$ is the function from $\{1,2,3\}$ to $\N$ with $s_1 (1) = 0$, $s_1 (2) = 1$, and $s_1 (3) = 0$ and that $s_2$ is the function from $\{1,2\}$ to $\N$ such that $s_2 (1) = 5$ and $s_2(2) = 6$; we abbreviate this by simply saying that $s_1 = (0,1,0)$ and $s_2 = (5,6)$.  Then $s_2 * s_1 = (0,1,0, 5, 0,1,0,6,0,1,0)$. 

{\bf The notation $\perp$:}  Suppose $s$ and $s^\prime$ are both functions from $\{1, 2, \ldots, r-1\}$ to $\N$.  We say that $s$ is {\em compatible} with $s^\prime$ if there exists a function $c$ from $\{1\}$ to $\N$ so that $s$ is a subsequence of $c*s^\prime$.   Otherwise we say that $s$ is incompatible with $s^\prime$ and write $s\perp s^\prime$. 

To illustrate, consider $s = (0,1,0)$ and $s^\prime = (0,0,1)$.  Then $s$ is compatible with $s^\prime$ because if $c = 0$, then $c * s^\prime = (0,0,1,0,0,0,1)$ and $(0,1,0)$ does occur as a subsequence of $(0,{ 0,1,0},0,0,1)$.  If $s^{\prime\prime} = (0,1,2)$, then $s^{\prime}$ is compatible with $s^{\prime\prime}$ (again let $c=0$), but $s \perp s^{\prime\prime} $, because $(0,1,0)$ can never be a subsequence of $(0,1,2,c,0,1,2)$.


Though not used in our arguments, it is worth noting, and is straightforward to check, that $s \perp s^\prime$ iff $s^\prime \perp s$. (It is important here that $s$ and $s^\prime$ have the same length.)

We now state the main point of this definition of incompatibility.  This fact will be crucial in the proof of Proposition 2.1.  Suppose $(r_n^\prime : n \in \N)$ and $(s_n^\prime : n \in \N)$ are cutting and spacer parameters associated to the symbolic rank-1 measure-preserving transformation $(Y, \nu, \sigma)$ and that $(w_n : n \in \N)$ is the generating sequence associated to those parameters.  
If $n$ is such that $r_n = r_n^\prime$ and $s_n \perp s_n^\prime$, then no element of $y \in Y$ contains an occurrence of $$w_n 1^{s_n (1)} w_n 1^{s_n(2)} \ldots 1^{s_n(r_n -1)} w_n$$ where each of the demonstrated occurrence of $w_n$ is expected.   

Indeed, suppose that that beginning at position $i$, some $y \in Y$ did have such an occurrence of $w_n 1^{s_n (1)} w_n 1^{s_n(2)} \ldots 1^{s_n(r_n -1)} w_n$.  The expected occurrence of $w_n$ beginning at $i$ must be completely contained in some expected occurrence of $w_{n+1}$, say that begins at position $j$.  We know that the expected occurrence of $w_{n+1}$ beginning at position $j$ contains exactly $r_n$-many expected occurrences of $w_n$.  Let $1 \leq l \leq r_n$ be such that the expected occurrence of $w_n$ beginning at position $i$ is the $l$th expected occurrence of $w_n$ beginning at position $j$.  If $l=1$, then $s_n = s_n^\prime$, which implies that $s_n$ is a subsequence of $c*s_n^\prime$ for any $c$.  If, on the other hand, $1< l \leq r_n$, then letting $c = s_n (r_n - l +1)$, we have that $s_n$ is a subsequence of $c*s_n^\prime$.  In either case this would result in $s_n$ being compatible with $s_n^\prime$.

\subsection{A general proposition guaranteeing non-isomorphism}

\begin{proposition}
\label{prop}
Let $(r_n: n \in \N)$ and $(s_n: n \in \N)$ be the cutting and spacer parameters for a symbolic rank-1 system $(X, \mu, \sigma)$ and let $(r^\prime_n: n \in \N)$ and $(s^\prime_n: n \in \N)$ be the cutting and spacer parameters for a symbolic rank-1 system $(Y, \nu, \sigma)$.  Suppose the following hold.  
\begin{enumerate}
\item  For all $n$, $r_n = r^\prime_n$ and $\displaystyle \sum_{0 < i < r_n} s_{n}(i) = \sum_{0 < i < r_n} s^\prime_{n}(i)$.
\item  There is an $S \in \N$ such that for all $n$ and all $0 < i < r_n$, $$s_n(i) \leq S \textnormal{ and } s^\prime_n(i) \leq S.$$
\item  There is an $R \in \N$ such that for infinitely many $n$, $$r_n \leq R \textnormal{ and } s_n \perp s^\prime_n.$$
\end{enumerate}   
Then $(X, \mu, \sigma)$ and $(Y, \nu, \sigma)$ are not measure-theoretically isomorphic. 
\end{proposition}

\begin{proof}
Let $(v_n: n \in \N)$ be the generating sequence associated to the cutting and spacer parameters $(r_n: n \in \N)$ and $(s_n: n \in \N)$.  Let $(w_n: n \in \N)$ be the generating sequence associated to the cutting and spacer parameters $(r^\prime_n: n \in \N)$ and $(s^\prime_n: n \in \N)$.  Condition (1) implies that for all $n$, $|v_n| = |w_n|$.  Now suppose, towards a contradiction, that $\phi$ is an isomorphism between $(X, \mu, \sigma)$ and $(Y, \nu, \sigma)$.    

First, choose $m \in \N$ so that $|v_m| = |w_m|$ is greater than the $S$ from condition (2).  Next, consider the positive $\mu$-measure set $$\phi^{-1} (E_{w_m, 0}) = \{x \in X: \phi (x) \text{ has an expected occurrence of $w_m$ at 0}\}.$$  Let $\mathbb{M} = \{n \in \N: \textnormal{$r_n \leq R $ and $s_n \perp s^\prime_n$}\}$, where $R$ is from condition (3), and note that $\mathbb{M}$ is infinite.  We can then find $n \in \mathbb{M}$ and $k \in \N$ such that $$\displaystyle \frac{\mu (E_{v_n, k} \cap \phi^{-1} (E_{w_m, 0}))}{\mu(E_{v_n, k}) } > 1- \frac{1}{R}.$$


One can loosely describe the above inequality by saying:  Most $x \in X$ that have an expected occurrence of $v_n$ beginning at position $k$ are such that $\phi (x)$ has an expected occurrence of $w_m$ beginning at position 0.

We say that an expected occurrence of $v_n$ in any $x \in X$ (say it begins at $i$) is a {\em good} occurrence of $v_n$ if $\phi (x)$ has an expected occurrence of $w_m$ beginning at position $i-k$.  In this case we say the good occurrence of $v_n$ beginning at $i$ {\em forces} the expected occurrence of $w_m$ beginning at position $i-k$.  Note that an expected occurrence of $v_n$ beginning at position $i$ in $x \in X$ is good iff $\sigma^{i} (x) \in E_{v_n, k} \cap \phi^{-1} (E_{w_m, 0})$, since $\phi$ commutes with $\sigma$.  A simple application of the ergodic theorem shows that $\mu$ almost every $x \in X$ satisfies $$\displaystyle \lim_{N \rightarrow \infty} \frac{|\{i \in [-N,N] : \textnormal{ $x$ has a good occurrence of $v_n$ at $i$}\}|}{|\{i \in [-N,N] : \textnormal{ $x$ has an expected occurrence of $v_n$ at $i$}\}|} >  1- \frac{1}{R}.$$
Since $ r_n \leq R$, this implies that $\mu$ almost every $x \in X$ contains an expected occurrence of $v_{n+1}$ such that each of the $r_n$-many expected occurrences of $v_n$ that it contains is good.  We say that such an occurrence of $v_{n+1}$ is {\em totally good}.

Let $x\in X$ and $i \in \Z$ be such that $x$ has a totally good occurrence of $v_{n+1}$ beginning at $i$.  There are $r_n$-many expected occurrences of $v_n$ in the expected occurrence $v_{n+1}$ beginning at $i$ and each of them forces an expected occurrence of $w_m$ in $\phi(x)$.  The first of these forced expected occurrences of $w_m$ in $\phi (x)$ begins at position $i - k$ and must be part of some expected occurrence of $w_n$, say it begins at position $i^\prime$.  We claim that, in fact, $\phi (x)$ must have an occurrence of $$w_n 1^{s_n (1)} w_n 1^{s_n(2)} \ldots 1^{s_n(r_n -1)} w_n$$ beginning at $i^\prime$, where each of the demonstrated occurrence of $w_n$ is expected.  This will contradict the fact that $s_n \perp s_n^\prime$.

Proving the claim involves an argument that is repeated $r_n -1$ times.  The next paragraph contains the first instance of that argument, showing that the expected occurrence of $w_n$ beginning at $i^\prime$ in $\phi (x)$ is immediately followed by $1^{s_n(1)}$ and then another expected occurrence of $w_n$, this one containing the second forced occurrence of $w_m$.  The next instance of the argument would show that the expected occurrence of $w_n$ beginning at $i^\prime + |w_n| + s_n(1)$ is immediately followed by $1^{s_n(2)}$ and then another expected occurrence of $w_n$, this one containing the third forced occurrence of $w_m$.  After the $r_n -1$ instances of that argument, the claim would be proven.



Here is the first instance of the argument:  We know that $\phi (x)$ has an expected occurrence of $w_n$ beginning at $i^\prime$ and that this expected occurrence of $w_n$ contains the expected occurrence of $w_m$ beginning at position $$i-k = (i^\prime) + (i - k - i^\prime).$$   Thus, by point (5) of the remark about expected occurrences in Section \ref{comments},  if $j \in \Z$ is such that $\phi (x)$ has an expected occurrence of $w_n$ beginning at position $j$, that occurrence completely contains an expected occurrence of $w_m$ beginning at position $j + (i - k - i^\prime)$. 



The expected occurrence of $w_n$ beginning at position $i^\prime$ must be followed by $1^t$ and then another expected occurrence of $w_n$, for some $0 \leq t \leq S$.  The expected occurrence of $w_n$ beginning at position $i^\prime + |w_n| + t$ must contain an expected occurrence of $w_m$ beginning at position $$(i -k) + |w_n| + t =  (i^\prime + |w_n| + t) + (i - k - i^\prime).$$ 
But we also know that the expected occurrence of $v_n$ beginning at position $i + |v_n| + s_n(1)$ in $x$ forces an expected occurrence of $w_m$ beginning at $i + |w_n| + s_n(1)$ in $\phi (x)$.  Since $0 \leq s_n(1),t \leq S$ it must be that $0 \leq |t - s_n(1)| \leq S$ and thus, since $|w_m| > S$, the expected occurrences beginning at positions $i + |w_n| + s_n(1)$ and $i + |w_n| + t$ must overlap.  Since distinct expected occurrences of $w_m$ cannot overlap, it must be the case that $s_n(1)=t$  (See point (4) of the remark about expected occurrences in Section \ref{comments}).  Thus, the expected occurrence of $w_n$ beginning at $i^\prime$ in $\phi (x)$ is immediately followed by $1^{s_n(1)}$ and then another expected occurrence of $w_n$, this one containing the second forced occurrence of $w_m$.
\end{proof}

\subsection{Proving the theorem}

We start this subsection with a comment and a simple lemma.  The comment is that if $(r_n: n \in \N)$ and $(s_n: n \in \N)$ are the canonical cutting and spacer parameters for a symbolic rank-1 measure-preserving transformation $(X, \mu, \sigma)$, then for all $n \in \N$, $s_{n+1}*s_n$ is not constant.  This is simply a restatement, using the notation $*$, of the last sentence in the remark about canonical generating sequences in Section \ref{comments}.

\begin{lemma}
\label{lemma}
Let $s_1$ be a function from  $\{1, 2, \ldots, r_1-1\}$ to $\N$ such that $s_1 \neq \overline{s_1}$.  Let $s_2$ be function from $\{1, 2, \ldots, r_2-1\}$ to $\N$ that is not constant.    Then $s_2 * s_1 \perp \overline{s_2 * s_1}$.
\end{lemma}

\begin{proof}
Suppose, towards a contradiction, that $s_2 * s_1$ is compatible with $\overline{s_2 * s_1}$.  Then there is function $c$ from $\{1\}$ to $\N$ so that $s_2 * s_1$ is a subsequence of $$c * (\overline{s_2 * s_1}) = c * (\overline{s_2} * \overline{s_1}) = (c * \overline{s_2}) * \overline{s_1}.$$  In other words, there is some $0 \leq k \leq r_2 \cdot r_1$ such that for all $0 < l < r_2 \cdot r_1$, $(s_2 * s_1)(l) = ((c * \overline{s_2}) * \overline{s_1})(k+l) $.  We now have two cases.

Case 1:  $k \equiv 0 \mod r_1$.  Then for all $0<i<r_1$, $$s_1(i) = (s_2 * s_1) (i) = ((c * \overline{s_2}) * \overline{s_1}) (k+i) = \overline{s_1} (i).$$  Thus $s_1 = \overline{s_1}$, which is a contradiction.

Case 2:  There is some $0<m<r_1$ such that $k+m \equiv 0 \mod r_1$.  For all $0 \leq d < r_2$, we have $(s_2 * s_1)(m + dr_1) = s_1 (m) $.  But also, $$(s_2 * s_1)(m + dr_1) = ((c * \overline{s_2}) * \overline{s_1}) (k + m + dr_1) = (c * \overline{s_2}) \left(\frac{k + m}{r_1} + d\right) .$$  This implies that the function $c*\overline{s_2}$ is constant (taking the value $s_1 (m)$) for $r_2$-many consecutive inputs.  This implies that $\overline{s_2}$ must be constant, which is a contradiction.

\end{proof}


We will now prove the non-trivial direction of the theorem.  

\begin{proof}
Let $(\tilde{r}_n: n \in \N)$ and $(\tilde{s}_n: n \in \N)$ be the canonical cutting and spacer parameters for a symbolic rank-1 measure-preserving transformation $(X, \mu, \sigma)$.  Suppose both parameters are bounded; let $\tilde{R}$ be such that for all $n \in \N$, $\tilde{r}_n \leq \tilde{R}$ and let $S$ be such that for all $n \in \N$ and all $0<i<\tilde{r}_n$, $\tilde{s}_n(i) \leq S$.   Also, assume that for infinity many $n$, $\tilde{s}_n \neq \overline{\tilde{s}_n}$.  To prove the non-trivial direction of the theorem, we need to show that $(X, \mu, \sigma)$ is not isomorphic to its inverse.

 Let $(u_n :n \in \N)$ be the generating sequence corresponding to the parameters $(\tilde{r}_n: n \in \N)$ and $(\tilde{s}_n: n \in \N)$.  We will now describe a subsequence $(v_n: n \in \N)$ of $(u_n:n \in \N)$ and let $(r_n:n \in \N)$ and $(s_n: n \in \N)$ be the cutting and spacer parameters corresponding to the generating sequence $(v_n: n \in \N)$ (which also gives rise to $(X, \mu, \sigma)$).  First, let $v_0 = u_0 =0$.  Now, suppose $v_{2n}$ has been defined as $u_{k}$.  Let $m>k$ be as small as possible so that $\tilde{s}_m \neq \overline{\tilde{s}_m}$, and define $v_{2n+1} = u_m$ and $v_{2n+2} = u_{m+3}$.  It is very important to note here that $$r_{2n+1} = \tilde{r}_{m+2} \cdot \tilde{r}_{m+1} \cdot \tilde{r}_{m} $$
and that $$s_{2n+1} =  \tilde{s}_{m+2} *  \tilde{s}_{m+1} *  \tilde{s}_{m}. $$  This has two important consequences.  First, we have that for $n \in \N$, $r_{2n+1} \leq \tilde{R}^3$.  By the remark before Lemma \ref{lemma}, we also have that $\tilde{s}_{m+3} * \tilde{s}_{m+2} $ is not constant and thus, by Lemma \ref{lemma}, $\tilde{s}_{m+3} * \tilde{s}_{m+2} *  \tilde{s}_{m+1} \perp \overline{\tilde{s}_{m+3} * \tilde{s}_{m+2} *  \tilde{s}_{m+1} }$; put another way, $s_{2n+1} \perp \overline{s_{2n+1}}$.

Now for each $n$, let $r_n^\prime = r_n$ and $s_n^\prime = \overline{s_n}$.  Let $(Y, \nu, \sigma)$ be the symbolic rank-1 transformation corresponding to the cutting and spacer parameters $(r_n^\prime : n \in \N)$ and $(s_n^\prime : n \in \N)$.  As mentioned in the remark on rank-1 inverses at the end of Section \ref{comments}, the transformation $(Y, \nu, \sigma)$ is isomorphic to the inverse of $(X, \mu, \sigma)$.  Thus to show that $(X, \mu, \sigma)$ is not isomorphic to its inverse, we can show that $(X, \mu, \sigma)$ and $(Y, \nu, \sigma)$ are not isomorphic.

To do this we will apply Proposition \ref{prop}.  We need to check that the following three conditions hold.
\begin{enumerate}
\item  For all $n$, $r_n = r^\prime_n$ and $\displaystyle \sum_{0 < i < r_n} s_{n}(i) = \sum_{0 < i < r_n} s^\prime_{n}(i)$.
\item  There is an $S \in \N$ such that for all $n$ and all $0 < i < r_n$, $$s_n(i) \leq S \textnormal{ and } s^\prime_n(i) \leq S.$$
\item  There is an $R \in \N$ such that for infinitely many $n$, $$r_n \leq R \textnormal{ and } s_n \perp s^\prime_n.$$
\end{enumerate}   

Condition (1) above follows immediately from the fact that for all $n \in N$, $r_n^\prime = r_n$ and $s_n^\prime = \overline{s_n}$.  Condition (2) follows from the fact that each $s_n(i)$ is equal to some $\tilde{s}_m (j)$ which is less than or equal to $S$.  Finally, to verify condition (3), let $R=\tilde{R}^3$ and note that, as remarked above, for all $n \in \N$ we have that $r_{2n+1} \leq \tilde{R}^3$ and $s_{2n+1} \perp \overline{s_{2n+1}}$.  We now apply Proposition \ref{prop} and conclude that $(X, \mu, \sigma)$ and $(Y, \nu, \sigma)$ are not isomorphic.  Thus $(X, \mu, \sigma)$ is not isomorphic to its inverse. 
\end{proof}

\bibliographystyle{amsalpha}

\end{document}